\documentclass[11pt, a4paper]{amsart}
\usepackage{amscd,amssymb,eucal,eufrak,mathrsfs,amsmath}
\usepackage{hyperref} 

\input xy
\xyoption{all}
\usepackage{epsfig}

\newtheorem{thm}{Theorem}[section]
\newtheorem{cor}[thm]{Corollary}
\newtheorem{lem}[thm]{Lemma}
\newtheorem{prop}[thm]{Proposition}

\newtheorem{question}[thm]{Question}     
  \newtheorem{f}[thm]{Fact}                              
\theoremstyle{definition}
\newtheorem{defin}[thm]{Definition}

\theoremstyle{remark}
\newtheorem{remark}[thm]{Remark}
\newtheorem{remarks}[thm]{Remarks}
\newtheorem{example}[thm]{Example}
\newtheorem{examples}[thm]{Examples}
\numberwithin{equation}{section}
\newtheorem{problem}[thm]{Problem}

\newtheorem{fact}[thm]{Fact}

\newcommand{\delete}[1]{} 
\newcommand{\nt}{\noindent}

\def\eps{{\varepsilon}}

\def\a{\alpha}

\def\eps{{\varepsilon}}

\newcommand{\g}{\gamma}

\newcommand{\dl}{\delta} 
\def\rem{ \overline{ \delta }}
\def\~dl{\overline{ \delta }}
\def\cl{\mathrm{cl}}

\def\To{\Longrightarrow}

\newcommand{\sk}{\vskip 0.2cm}

\newcommand{\ben}{\begin{enumerate}}

\newcommand{\een}{\end{enumerate}}
\newcommand{\bit}{\begin{itemize}}

\newcommand{\eit}{\end{itemize}}

\def\R {{\mathbb R}}
\def\N {{\mathbb N}}
\def\Z {{\mathbb Z}}
\def\Q {{\mathbb Q}}

\def\U{\mathcal{U}} 

\def\T {{\mathbb T}}

\def\obr{^{-1}}

\def\RUC{{\hbox{RUC}_G}}

\def\Iso{{\mathrm{Iso}}\,}
\def\Aut{{\mathrm Aut}\,}

\def\Homeo{{\mathrm{Homeo}}\,}

\newcommand{\lan}{\langle}
\newcommand{\ran}{\rangle}

\def\wrt{with respect to }

\def\QED{\nobreak\quad\ifmmode\roman{Q.E.D.}\else{\rm Q.E.D.}\fi}

\begin{document}

\title[]{Maximal equivariant compactifications} 


\dedicatory{Dedicated to the centennial of Yu.M. Smirnov's birth}

\author[]{Michael Megrelishvili}
\address{Department of Mathematics,
	Bar-Ilan University, 52900 Ramat-Gan, Israel}
\email{megereli@math.biu.ac.il}
\urladdr{http://www.math.biu.ac.il/$^\sim$megereli}


\date{2022, January}  


\thanks{{\it 2020 AMS classification:} 54H15, 54D35, 54F05} 
\thanks{This research was supported by a grant of the Israel Science Foundation (ISF 1194/19) 
 and also by the Gelbart Research Institute at the Department of Mathematics, Bar-Ilan  University}

\begin{abstract}  
Let $G$ be a locally compact group. Then for every $G$-space $X$ the maximal $G$-proximity $\beta_G$ can be characterized by the maximal topological proximity $\beta$ as follows:
$$
A \ \overline{\beta_G} \ B \Leftrightarrow \exists V \in N_e \ \ \ VA \ \overline{\beta} \ VB.
$$
Here, $\beta_G \colon X \to \beta_G X$ is the maximal $G$-compactification of $X$ (which is an embedding for locally compact $G$ by a result of J. de Vries), $V$ is a neighborhood of $e$ and $A \ \overline{\beta_G} \ B$ means that the closures of $A$ and $B$ do not meet in $\beta_G X$.  

Note that the local compactness of $G$ is essential. 
This theorem comes as a corollary of a general result about maximal $\mathcal{U}$-uniform $G$-compactifications for a useful wide class of uniform structures $\mathcal{U}$ on $G$-spaces for not necessarily locally compact groups $G$. 
It helps, in particular, to derive the following result. 
Let $(\mathbb{U}_1,d)$ be the Urysohn sphere and  $G=\Iso(\mathbb{U}_1,d)$ is its isometry group with the pointwise topology. Then for every pair of subsets $A,B$ in $\mathbb{U}_1$, we have 
$$
A \ \overline{\beta_G} \ B \Leftrightarrow \exists V \in N_e  \ \ \  
d(VA,VB) > 0.
$$  
More generally, the same is true for any  $\aleph_0$-categorical metric $G$-structure $(M,d)$, where $G:=\Aut(M)$ is its automorphism group. 
\end{abstract} 

\maketitle  
\setcounter{tocdepth}{1}
 \tableofcontents

\section{Introduction}

A \emph{topological
	transformation group} ($G$-\emph{space}) is a continuous action of a topological group
$G$ on a topological space $X$.  
 Compactifiability of Tychonoff topological spaces 
means the existence of topological embeddings into compact
Hausdorff spaces. For the compactifiability of $G$-spaces we require, in
addition, the continuous extendability of the original action.
Compactifiable $G$-spaces are known also as $G$-\emph{Tychonoff} \emph{spaces}.  

Compactifications of $G$-spaces is a quite an active research field. 
We do not intend here to give 
a comprehensive bibliography but try to refer the interested readers to some publications, where 
$G$-compactifications play a major role. See, for example, R. Brook \cite{Br}, 
J.~de Vries \cite{Vr-can75,Vr-book75,Vr-Embed77,Vr-loccom78,Vr-LinComp82,Vr-Revis00}, 
Yu.M. Smirnov \cite{Smirnov76,Smirnov77,Smirnov82,Smirnov-geom94},  Antonyan--Smirnov \cite{AS}, Smirnov--Stoyanov \cite{SmirnStoy}, L. Stoyanov \cite{Sto1,Sto2}, M. Megrelishvili \cite{Me-EqNorm83,Me-EqComp84,Me-diss85,Me-Ex88,Me-sing89,Me-ec94,Me-opit07,Me-b}, 
 Dikranjan-Prodanov-Stoyanov \cite{DPS}, Megrelishvili--Scarr \cite{MeSc98}, 
V. Uspenskij \cite{Usp-Dugundji}, S. Antonyan \cite{Ant-emb}, Gonzalez--Sanchis \cite{GonzSan06}, V. Pestov  \cite{Pe-nbook,Pest-Smirnov}, J. van Mill \cite{Mill09}, A. Sokolovskaya \cite{Sok}, Google--Megrelishvili \cite{GM-prox10},  Kozlov--Chatyrko \cite{Koz-Chat10}, 
 N. Antonyan, S. Antonyan and M. Sanchis \cite{AAS}, K. Kozlov \cite{Koz-Rectang12, Koz-spectr13,Koz-homog13,Koz-eq22}, N. Antonyan \cite{NAnt17}, Karasev--Kozlov \cite{Kar-Koz}, Ibarlucia--Megrelishvili \cite{IbMe20}   
(and many additional references in these publications). 

\sk 
Compactifications of a Tychonoff space $X$ can be described in
several ways: 

\sk 
\bit
\item  Banach subalgebras of $C^b(X)$
(Gelfand-Kolmogoroff 1-1 correspondence));
\item
Completion of totally bounded 
uniformities on $X$ (Samuel compactifications);
\item
Proximities on $X$ (Smirnov compactifications).
\eit

It is well known (see for example \cite{Br,AS,Vr-Embed77,Vr-loccom78,Me-EqComp84})
that the first two correspondences admit dynamical generalizations
in the category of $G$-spaces.
Instead of continuous bounded functions, we should use special
subalgebras of generalized right uniformly continuous functions (in other terminology, \textit{$\pi$-uniform functions})  
and instead of precompact uniformities, we need
now precompact \emph{equiuniformities} (Definition
\ref{d:qb}).
 
For every Tychonoff $G$-space $X$, the algebra $\RUC(X)$ of all right uniformly continuous bounded functions on $X$ induces the corresponding Gelfand (maximal ideal) space $\beta_G X  \subset \RUC(X)^*$ and the \textit{maximal $G$-compactification} 
$$\beta_G \colon X \to  \beta_G X.$$

For locally compact groups $G$, all Tychonoff $G$-spaces admit proper compactifications, as was established by de Vries \cite{Vr-loccom78}. So, in this case, the map $\beta_G$ is a topological embedding. However, in general it is not true. Resolving a question of de Vries \cite{Vr-can75}, we proved in \cite{Me-Ex88} that there exist noncompactifiable $G$-spaces (even for Polish group actions on Polish spaces). 

Moreover, 
answering an old problem due to Smirnov, an extreme example was found by V. Pestov \cite{Pest-Smirnov} by constructing a countable metrizable group $G$ and a countable metrizable non-trivial $G$-space $X$ for which every equivariant compactification is a singleton.

\sk 
One of the most general (and widely open) attractive problems is 
\begin{problem} \label{pr:maximal} 
	Clarify the structure of maximal $G$-compactifications $\beta_G X$ of remarkable naturally defined $G$-spaces $X$.
\end{problem} 
First of all, note that  $\beta_G X$ (for nondiscrete $G$) usually is essentially ``smaller" than $\beta X$. For instance, let $G$ be a metrizable topological group which is not precompact. Then the canonical action $G \times \beta G \to \beta G$ is continuous iff $G$ is discrete (Proposition \ref{ex:beta}). 


\begin{problem} \cite[Question 1.3(b)]{IbMe20} 
	Study the greatest $G$-compactification $\beta_G \colon X \to \beta_G X$ of (natural) Polish $G$-spaces $X$. In particular: when is $\beta_G X$ metrizable?
\end{problem}

Recall that the \^{C}ech–Stone compactification $\beta X$ of any metrizable non-compact space $X$ cannot be metrizable. In contrast, 
for several naturally defined ``massive actions", $\beta_G X$ might be  metrizable;  sometimes even having a nice transparent geometric presentation. Perhaps the first 
example of this kind was a beautiful result of L. Stoyanov \cite{Sto1, Sto2}. He established that the greatest $U(H)$-compactification of the unit sphere $S_H$ in every infinite dimensional Hilbert space $H$ is the weakly compact unit ball, where $G=U(H)$ is the unitary group of $H$ in its standard strong operator topology. 

One of the important sufficient conditions when a $G$-space $X$ is $G$-compactifiable 
is the existence of a $G$-invariant metric on $X$. This was proved first by Ludescher--de Vries \cite{Lud-Vr80}.   
Another possibility to establish that such $(X,d)$ is $G$-Tychonoff is to observe  that in this case \textit{Gromov compactification} $\g \colon (X,d) \to \g(X)$ is a $G$-compactification which is a $d$-\textit{uniform} topological embedding; see the explanation in \cite{Me-opit07} using the RUC property of the distance functions $x \mapsto d(x,\cdot)$ (one may assume that $d$ is bounded). 

Pestov raised several questions in \cite{Pest-Smirnov} about a possible coincidence between the maximal $G$-compactification and the Gromov compacfification for some natural geometrically defined isometric actions (Urysohn sphere and Gurarij sphere, among others). 
These problems were studied recently in 
\cite{IbMe20} (with a positive answer in the case of the Urysohn sphere and a  negative answer for the Gurarij sphere). 






\sk 
\begin{remark} \label{r:smallGcomp} 
We collect here some old and new nontrivial concrete examples when $\beta_G X$ is metrizable, usually admitting also a geometric realization.  
\begin{enumerate} 
	\item (L. Stoyanov \cite{Sto1,Sto2,DPS}) Let $X:=S_H$ be the unit sphere of the infinite-dimensional separable Hilbert space $H$ with the unitary group $G:=U(H)$. Then $\beta_G X$ is the weak compact unit ball $(B_H,w)$ of $H$. 
	\item \cite{IbMe20} \textit{Urysohn sphere} $(\mathbb{U}_1,d)$ with its isometry group 
	$G=\Iso (\mathbb{U}_1)$. Then $\beta_G \mathbb{U}_1$ can be identified with its Gromov compactification. 
	Moreover, $\beta_G \mathbb{U}_1$ can be identified with the compact space  $K(\mathbb{U}_1)$ of all Katetov functions on $X$. 
	\footnote{K. Kozlov proved in \cite{Koz-eq22} that $\beta_G \mathbb{U}_1$ is homeomorphic to the Hilbert cube.}  
	\item \cite[Theorem 4.11]{IbMe20} The maximal $G$-compactification of the unit sphere $S_{\mathfrak{G}}$ in the Gurarij Banach space $\mathfrak{G}$ (where $G$ is the linear isometry group) is metrizable and does not coincide with its Gromov compactification. $\beta_G (S_{\mathfrak{G}})$ can be identified with the compact space  $K^1_C(\mathfrak{G})$ of all normalized Katetov convex functions on $\mathfrak{G}$. These results are strongly related to some properties of the  Gurarij space studied by I. Ben Yaacov \cite{BenYa} and Ben Yaacov--Henson \cite{BenYaHenson}. 
	\item 
	(Proved in \cite[Theorem 4.14]{IbMe20} thanks to an observation of Ben Yaacov) 
	 Let $B_p$ be the unit ball of the classical Banach space 
	$V_p:=L_p[0, 1]$ for $1 \leq p < \infty$, $p  \notin 2\N$. Then for the linear isometry group $\Iso_l(V_p)$, the maximal 
	$G$-compactification $\beta_G B_p$ is the Gromov compactification of the metric space $B_p$. 
	\item \cite[Theorem 4.4]{IbMe20} For every $\aleph_0$-categorical metric structure $(M,d)$,  
	the maximal $G$-compactification of $(M,d)$, with $G:=\Aut(M)$,
	can be identified with the space $S_1(M)$ of all 1-types over $M$ (and, in particular, is metrizable).
	\item (see Examples \ref{ex:massive} below) Let $X=(\Q,\leq)$ be the rationals with the usual order but equipped with the discrete topology. Consider any dense subgroup $G$ of the automorphism group $\Aut(\Q,\leq)$ with the pointwise topology (for instance, Thompson's group $F$). In this case 
	$\beta_G X$ is a metrizable linearly ordered compact $G$-space, the actions and $\beta_G \colon X \to \beta_G$ are order preserving, where $\beta_G X$ is an inverse limit of finite linearly ordered 
	spaces $\Q/St_F$, where $F \subset \Q$ is finite, $St_F$ is the stabilizer subgroup and $\Q/St_F$ is the orbit space.  
\end{enumerate}
\end{remark}

\sk 
Whenever $\a \colon X \to Y$ is a compactification, one of the natural questions is which subsets $A, B$ of $X$ are ``far" with respect to $\a$. 
This means that the closures of their images do not meet in $Y$. 
This is the most basic idea of classical \textit{proximity spaces}.  
See  Section \ref{s:Sm} for a description of the role of proximities and Smirnov's Theorem. 
This theorem shows that for every $G$-compactification $\a \colon X \to Y$ and $\a$-far subsets $A,B$ of $X$, there exists a sufficiently small neighborhood $V \in N_e$ of the identity in $G$ such that $VA,VB$ are also $\a$-far. 

A natural question arises about the converse direction:  when 
does this condition guarantee 
that we have a $G$-compactification? 
We show that this holds for proximities induced by a certain rich class of uniform structures on $G$-spaces (see Theorem \ref{t:G-proxOfQuasib}). 
This leads to  
one of the main results of this paper which is to describe 
maximal equivariant compactification of locally compact group actions (Theorem \ref{c:MaxGproximityLocCompG}). The local compactness of $G$ is necessary. Indeed, 
there exist a Polish $G$-compactifiable $G$-space $X$ with a Polish acting group $G$ and $G$-invariant closed $G$-subsets $A,B$ in $X$ such that $A \beta_G B$ (see Example \ref{ex:ex} and Remark \ref{r:ex}). 

\sk 
A more special general problem is 

\begin{problem} \label{pr:d} 
	For which metric $G$-spaces $(X,d)$ 
	is the following condition satisfied for every subsets $A,B$ in $X$
	  $$
	  A \ \overline{\beta_G} \ B \Leftrightarrow \exists V \in N_e  \ \ \  
	  d(VA,VB) > 0.
	  $$ 
\end{problem}

\sk 
Using a result from \cite{IbMe20},  
we positively answer Problem \ref{pr:d} for an important class of metric $G$-spaces. Namely, for  
$\aleph_0$-\textit{categorical metric structures} $(M,d)$, where $G:=\Aut(M,d)$ is  its automorphism group. In particular, this is true for the Urysohn sphere $\mathbb{U}_1$ (Theorem \ref{t:Urys}).

\sk 
\nt \textbf{Acknowledgment:} It is a great honor for me to say that Yu.M. Smirnov led me to the world of equivariant topology. I am grateful to T. Ibarlucia and V. Pestov for their influence and inspiration. I thank the organizers of the Conference Smirnov-100 for their work to provide such an important conference. 

\sk 
\section{Proximities and equivariant Smirnov's Theorem} 
\label{s:Sm} 

\subsection*{Proximities and proximity spaces.}
\label{sb:prox}

In 1908, F. Riesz  first formulated a set of axioms to
describe the notion of closeness of pair of sets.
The most useful version of proximity was introduced and studied by
V.A. Efremovich \cite{Efrem}.  
We follow the setting of \cite{NW}. 


\begin{defin}\label{df:prx} 
	Let $X$ be a nonempty set and $\delta$ be a relation in the set of all
	its subsets. We write $A \delta B$ if $A$ and $B$ are
	$ \delta$-related and
	$A \rem B$ if not. The relation $\dl$
	will be called a \emph{proximity} on $X$ provided that the
	following conditions are satisfied:
	\bit
	\item [(P1)]
	$A \cap B \neq \emptyset$ implies $A \delta B$.
	
	\item [(P2)]
	$A \delta B$ implies $B \delta A$;
	
	\item [(P3)]
	$A \delta B$ implies $A \neq \emptyset$;
	\item [(P4)]
	$A \delta (B\cup C)$ iff $A \delta B$ or $A \delta C$;
	
	\item [(P5)]
	If $A \rem B$ then there exist $C \subset X$ such that
	$A \rem C$ and $(X \backslash C) \rem B$.
	\eit 
\end{defin}

A pair $(X, \delta)$ is called a \emph{proximity space}. 
Two sets
$A,B \subset
X$ are \emph{near} (or \emph{proximal}) in $(X, \delta)$ if $A \delta B$ and \emph{far} 
 (or, \emph{remote}) if $A \~dl B$.  
 We say a subset $A \subset X$ is \emph{strongly contained} in $B \subset X$ \wrt $\dl$ (or, $B$ is a $\delta$-\textit{neighborhood} of $A$) 
 if $A \~dl (X\setminus B)$ and write: $A \Subset B.$ 
 In Definition \ref{df:prx} one can replace (P5) by the
 following axiom:
 
 \begin{itemize}
 	\item [(P5$'$)]  
 	If $A \~dl B$ then there exist subsets $A_1$ and $B_1$ of $X$ such that $A
 	\Subset A_1, \ B\Subset B_1$ and $ A_1 \cap B_1 = \emptyset.$
 \end{itemize}
  
 \sk 
Every proximity space $(X,\dl)$ induces a topology
$\tau:=top(\dl)$ on $X$ by the closure operator:
$$
\cl_{\dl}(A):=\{x \in X: \ x\dl A\}.
$$
The topology $top(\dl)$ is Hausdorff iff the following condition
satisfied: 

\sk

\begin{itemize}
	\item [(P6)]  
	If $x, y \in X$ and $x \dl y$ then $x=y$. 
\end{itemize}

\sk

Every (separated) proximity space $(X,\dl)$ is completely regular
(resp., Tychonoff)
\wrt the topology $top(\dl)$.
A proximity $\dl$ of $X$
is called \emph{continuous} (or, more precisely, a 
$\tau$-\emph{continuous proximity}) if $top(\dl) \subset \tau$.
In the case of $top(\dl) = \tau,$
we say that $\dl$
is a \emph{compatible} proximity on the topological space $(X,\tau).$

%

Like compactifications, the family of all proximities on
$X$ admits  a natural partial order. 
   A proximity $\dl_{1}$ \emph{dominates} $\dl_{2}$ (and write
    $\dl_{2}  \preceq \dl_{1}$) iff for every $A \dl_1 B$
    we have $A \dl_2 B.$

%
%

\begin{example} \label{e:max0} \ 
	\ben 
	\item
	Let $Y$ be a compact Hausdorff space. Then there exists a unique compatible proximity
	on the space $Y$ defined by
	$$
	A \dl B \Leftrightarrow cl(A) \cap cl(B) \neq
	\emptyset.
	$$
	
	\item
	Let $X$ be a Tychonoff space. The relation $\dl_{\beta}$
	defined by
	$$
	A \beta B \Leftrightarrow
	\nexists f \in C(X) \text { such that } f(A)=0
	\text { and } f(B)=1
	$$
	is a proximity which
	corresponds to the greatest compatible uniformity on $X$. The proximity $\beta$ comes from the \^{C}ech-Stone compactification $\beta \colon X \to \beta X$.
	\item A Hausdorff
	topological space $X$ is normal iff the relation
	$$
	A \delta_n B \quad \text{iff} \quad
	cl(A) \cap cl(B) \neq \emptyset
	$$
	defines a proximity relation on the set $X$. Then $A \delta_n B \Leftrightarrow A \beta B$. 
	\een
\end{example}

\subsection*{Smirnov's Theorems}

Let $\a\colon X \to Y$ be a compactification. Denote by $\dl_{\a}$ the
corresponding initial proximity on $X$ defined via the canonical proximity $\dl_Y$ of $Y$. More precisely, for subsets
$A,B$ of $X$ we define $A \bar{\dl_{\a}} B$ if $\a(A) \ 
\overline{\dl}_Y  \ \a (B)$, i.e., if
$cl(\alpha(A)) \cap cl(\alpha(B))=\emptyset$.

Conversely every continuous proximity $\dl$ on a topological space
induces a totally bounded uniformity $\mathcal{U}_{\dl}$.  
Now the completion gives \textit{Smirnov's compactification} $c_{\dl} \colon (X,\dl) \to c_{\dl} X$. It is equivalent to the Samuel
compactification \wrt the uniformity $\mathcal{U}_{\dl}$. This leads
to a description of compactifications in terms of proximities (see, for example,  \cite{Smirnov52, NW, Eng}).

\begin{f} \label{th:Sm.thm}
	(\texttt{{Smirnov's classical theorem}})
	Let $X$ be a topological space. 
	Assigning to any compactification $\alpha \colon X \to Y$
	the proximity $\dl_{\alpha}$ on $X$ gives rise to a
	natural one-to-one order preserving correspondence
	between all compactifications of $X$
	and all continuous proximities on the space $X.$
\end{f}

In the
case of $G$-spaces, it was initiated by Smirnov
himself, extending in \cite{AS} his old classical purely topological results from Tychonoff spaces to the case of group actions.

For proximities of $G$-compactifications we simply say $G$-\textit{proximity}.
\begin{f} \label{th:Sm.thmG} 
	(\texttt{{Smirnov's theorem for group actions}})
	In Smirnov's bijection (Fact \ref{th:Sm.thm}), $G$-proxmities are exactly proximities $\dl$ which satisfy the following two conditions:
	\begin{enumerate}
		\item ($G$-\emph{invariant}) \ $gA \ \dl \ gB$ for every $A \ \dl \ B$ and $g \in G$;
		\item (compatible with the action)	if $A \rem B$ then there exists $U \in N_e$ such that $UA \cap UB =\emptyset$.
	\end{enumerate} 
\end{f}


\begin{remark} \label{r:semigr} 
	The compatability condition (2) can be replaced by the following (formally stronger) assumption: 
	\begin{itemize}
		\item [(2$^{str}$)] if $A \rem B$ then there exists $V \in N_e$ such that $VA \rem VB$.
	\end{itemize}
In order to see that (2) implies (2$^{str}$), apply the axiom (P5$'$) (from Definition \ref{df:prx}). Then for $A \rem B$ there exist subsets $A_1$ and $B_1$ of $X$ such that $A \Subset A_1, \ B\Subset B_1$ and $ A_1 \rem B_1$. 
By (2) there exists $V \in N_e$ such that $VA \subset A_1, VB \subset B_2$. Therefore, $VA \rem VB$. 

\sk
The same can be derived also by results of \cite[Section 5.2]{GM-prox10}, where a natural generalization of Smirnov's theorem (Fact \ref{th:Sm.thmG}) for \textit{semigroup actions} was obtained. 
\end{remark}

Note that the $G$-invariantness of $\dl_{\a}$ guarantees that the $G$-action on $X$ can be extended to a $G$-action on the compactification $Y$ such that all $g$-translations are continuous. That is, we have a continuous action $G_{discr} \times Y \to Y$, where $G_{discr}$ is the group $G$ with the discrete topology 
(however, see Fact \ref{f:suff}.6 and Theorem \ref{t:OrdComp} below).



\begin{remark}
	Let $X$ be a locally compact Hausdorff space. Then the
	following relation 
	$$
	A \~dl_a B \Leftrightarrow cl(A) \cap cl(B) = \emptyset \ \text{where either} \
	cl(A) \ \text{or} \ cl(B) \ \text{is compact}
	$$
	defines a compatible proximity on $X$ which
	suits the (one-point) Alexandrov compactification.
	If $X$ is a $G$-space then this is a $G$-proximity. This explains Fact \ref{f:suff}.2 below. 
\end{remark}

\begin{remark} \label{r:right} 
	Let $G/H$ be a coset $G$-space with respect to the left
	action $\pi \colon G \times G/H \to G/H$ and a closed subgroup $H$. 
	The relation
	$\dl_R$
	defined by
	$$
	A \dl_R B \Leftrightarrow
	\exists U \in N_e(G): \ UA \cap B \neq  \emptyset
	$$
	is a compatible proximity on $G/H$. 
	In fact, it is a $G$-proximity (see Fact \ref{f:suff}.1 below). 
\end{remark}

\sk 
\subsection*{Uniform spaces and the corresponding proximity}

Recall the following standard lemma about the basis of a 
uniform structure (defined by the entourages -- reflexive binary relations) in the sense of A. Weil.  

\begin{lem} \label{l:UnifBasis} \emph{(see, for example, \cite[Prop. 0.8]{RD})}  
	An abstract set $\mathcal{B}$ of entourages on $X$ is a basis of some uniformity $\U$ iff the following conditions are satisfied:
	\begin{itemize}
		\item [(1)] $\forall \eps \in \mathcal{B}$ \ $\Delta_X \subset \eps$; 
		\item [(2)] $\forall \eps \in \mathcal{B}$ \ $\exists \delta \in \mathcal{B}$ \ $\delta \subset \eps^{-1}$;
		\item [(3)] $\forall \eps, \delta \in \mathcal{B}$ \ $\exists \gamma \in \mathcal{B}$ \ $\gamma \subset \eps \cap \delta$ ($\mathcal{B}$ is a filterbase); 
		\item [(4)] $\forall \eps \in \mathcal{B}$ \ $\exists \delta \in \mathcal{B}$ \ $\delta \circ \delta \subset \eps$. 
	\end{itemize}	
\end{lem}

The corresponding induced uniformity is just the filter generated by $\mathcal{B}$.
Each uniformity $\U$ on $X$ defines a topology $top(\U)$ on $X$ as follows: a subset $A \subset X$ is open iff for each $a \in A$ there exists $\eps \in \U$ such that $\eps(a) \subset A$, where 
$\eps(x):=\{y \in X: (x,y) \in \eps\}.$
 
$top(\U)$ is Hausdorff iff \ $\cap \{\eps: \eps \in \mathcal{B}\} =\Delta.$ If otherwise not stated, we always consider only Hausdorff completely regular (i.e., Tychonoff) topological spaces, Hausdorff uniformities and  \textit{proper compactifications} $\a \colon X \to Y$ (i.e., $\a$ is an embedding). 

\begin{defin} \label{d:ind-un}
	Let $\mathcal{U}$ be a uniformity on $X.$ Then the relation $\dl_{\mathcal{U}}$ defined
	by \label{l:prox_for_unif} 
	$$ 
	A\dl_{\mathcal{U}}B \Longleftrightarrow
	\eps \cap (A\times B) \neq \emptyset \ \ \ \forall\eps \in \mathcal{U}
	$$
	is a proximity on $X$ which is called the \emph{proximity induced by the uniformity
		$\mathcal{U}.$}	
\end{defin}

Always, $top(\mathcal{U})=top(\dl_{\mathcal{U}})$. 
Conversely, every proximity $\dl$ on a topological space
defines canonically a totally bounded compatible uniformity
$\mathcal{U}_{\dl}$. 

We say that a proximity $\nu$ on $X$ is $\U$-\textit{uniform} if $\nu \preceq \dl_{\mathcal{U}}$.

\sk
\section{Uniform $G$-spaces}

\begin{defin} \label{d:equic} 
	Let $\pi\colon G \times X \to X$ be a group action.
	A uniformity $\U$ on $X$ is:
	\begin{enumerate} 
		\item \emph{equicontinuous} if (the set of all translations is equicontinuous)
		$$
		\forall x_0 \in X \ \forall \eps \in \U \ \exists O \in N_{x_0} \    \ \ (gx_0,gx  ) \in \eps \ \ \ \ \forall x \in O \ \forall g \in G;
		$$
		\item \emph{uniformly equicontinuous} if (the set of all translations is uniformly equicontinuous)
		$$
		\forall \eps \in \U \ \exists \delta \in \U \    \ \ (gx,gy  ) \in \eps \ \ \ \ \forall (x,y) \in \delta \ \forall g \in G;
		$$
		\item if the conditions (1) or (2) are true for a subset $P \subseteq G$ then we say that $P$ acts equicontinuously or uniformly equicontinuously, respectively.
	\end{enumerate}
\end{defin}


If $d$ is a $G$-invariant metric on $X$, then the corresponding uniform structure $\U(d)$ is a very natural case of a uniformly equicontinuous uniformity. 

\begin{defin} \label{d:qb} 
	Let $\pi\colon  G \times X \to X$ be an action of a topological group $G$ on a set $X$ and  $\U$ is a uniform structure on $X$.
	\ben 
	\item We say that $\U$ is \textit{saturated} if every translation $\pi^g \colon X \to X$ is uniformly continuous (equivalently, if $g \eps \in \U$ for every $\eps \in \U$ and $g \in G$). 
	\item \cite{Br} $\U$ 
	is \textit{bounded} (or, \emph{motion equicontinuous}) if 
	$$
	\forall \eps \in \U \ \exists V \in N_{e} \    \ \ (vx,x) \in \eps \ \ \  \forall v \in V.
	$$
	\item $\U$ is \textit{equiuniform} if it is bounded and saturated. Notation: $(X,\U) \in \rm{EUnif}^G$.
	\item (introduced in \cite{Me-diss85,Me-EqComp84}
$\U$ is \emph{quasibounded} (or, $\pi$-\textit{uniform at $e$}) if 
	$$
	\forall \eps \in \U \ \exists V \in N_{e} \  \exists \delta \in \U   \ \ 
	(vx,vy) \in \eps \ \ \ \ \forall (x,y) \in \delta \ \forall v \in V.
	$$
	\emph{$\pi$-uniform} will mean quasibounded and saturated (or, $\pi$-uniform at every $g_0 \in G$); meaning that  
	$$
	\forall \eps \in \U \ \exists V \in N_{g_0} \  \exists \delta \in \U   \ \ (vx,vy) \in \eps \ \ \ \ \forall (x,y) \in \delta \ \forall v \in V.
	$$  
	\nt Notation: $(X,\U) \in \rm{Unif}^G$.
	\een
\end{defin}

\sk 

Every compact $G$-space (with its unique uniform structure) is equiuniform.  
$\rm{EUnif}^G \subset \rm{Unif}^G$ (by the ``$3 \eps$-argument") and both are closed under $G$-subspaces, the supremum of uniform structures, uniform products 
 and completions (Fact \ref{f:G-compl}).
Quasibounded uniformities give simultaneous generalization of uniformly equicontinuous and bounded uniformities on a $G$-space.  
The class $\rm{Unif}^G$ is characterized by Kozlov \cite{Koz-Rectang12} in terms of \textit{semi-uniform maps} (in the sense of J. Isbell) on products. Bounded 
uniformities and $G$-compactifications play a major role in the book of V. Pestov \cite{Pe-nbook}.




\begin{remarks} \label{r:simple-qb} \ 
	\begin{enumerate}
		\item \cite{Me-EqComp84} There exists a natural 1--1 correspondence between proper $G$-compactifications of $X$ and totally bounded equiuniformities on $X$. 
		\item If the action on $(X,\U)$ is uniformly equicontinuous (e.g., every isometric action) 
		then $(X,\U) \in \rm{Unif}^G$.
		\item More generally: assume that there exists a neighborhood $V \in N_e$ such that $V$ acts uniformly
		equicontinuously on $(X,\U)$ and the action of $G$ on $X$ is $\U$-saturated.
		Then $(X,\U) \in \rm{Unif}^G$.
		
		\item 	Let $G$ and $X$ both are topological groups and $\a \colon G \times X \to X$ be a continuous action by group automorphisms. Then  $(X,\U) \in \rm{Unif}^G$, where $\U$ is right, left, two-sided or Roelcke uniformity on $X$. 
		\item Not every quasibounded action is bounded. 
		For example, the natural linear action of the circle group $\T$ on the metric space $\R^2$ is
		quasibounded (even, uniformly equicontinuous) because it preserves the metric but not bounded. 
	\end{enumerate} 
\end{remarks}

\sk 
\begin{prop} \label{equiun->S-prox}
	Let $\U$ be an equiuniformity on a $G$-space $X$. Then $\dl_{\U}$
	is a $G$-proximity (hence, the corresponding Smirnov compactification and also the Samuel compactification $s \colon (X,\U) \to sX$ are proper $G$-compactifications).
\end{prop}
\begin{proof}
	Let $A \~dl_{\U} B.$ By Definition \ref{d:ind-un}
	there exists an entourage $\eps \in \U$,
	such that  $(A \times B) \cap \eps = \emptyset$.
	Fix $g_0 \in G$. Then we claim that there exist
	$\eps' \in \U$ and a neighborhood $V$ of $g_0$ in $G$ such that
	$V^{-1}A$ and $V^{-1}B$ are $\eps'$-far (this means that 
	$V^{-1}A \ \~dl_{\U} \ V^{-1}B$). 
	
	Since $\U$ is an equiuniformity, it follows that for $\eps \in \U$ we can choose
	$\eps' \in \U$ and $V \in N_{g_0}\,$ such that
	\begin{gather} \label{fr:prx.unif}
		(x,y) \in \eps' \To (g_1x,g_2y) \in \eps, \ \  \forall \ g_1,g_2 \in V.
	\end{gather}
	Now we claim that $(V^{-1}A \times V^{-1}B) \cap \eps' =\emptyset.$
	Assuming the contrary, we get
	$$
	\eps' \cap (V\obr A \times V\obr B) \neq \emptyset \To
	\exists \ (x,y) \in (V\obr A \times V\obr B ): (x,y) \in \eps'.
	$$
	Therefore by definition of $V \obr A$ and $ V \obr B$, we
	conclude 
	$$
	\exists \ g',g'' \in V: \ \ (g'x,g''y) \in A \times B.
	$$
	On the other hand by Formula \ref{fr:prx.unif} for $g',g'' \in V$, 
	we have 
	$$
	(x,y) \in \eps' \To (g'x,g''y) \in \eps.
	$$
	This means
	$$
	\forall \ \eps \in \U: \ (g'x,g''y) \in (A \times B) \cap \eps. 
	$$
	
	Hence $(A \times B) \cap \eps \neq \emptyset$, a contradiction. 
	%
\end{proof}

\begin{f} \label{f:G-compl} \emph{(Completion theorem \cite{Me-ec94})}  
	Let $(X,\U) \in \rm{Unif}^G$. Then the action $G \times X \to X$ continuously can be extended to the action on the completion $\widehat{G} \times \widehat{X} \to \widehat{X}$, where $(\widehat{X},\widehat{\U}) \in \rm{Unif}^{\widehat{G}}$ and $\widehat{G}$ is the Raikov completion of $G$. 
\end{f}

\begin{cor} \label{c:DenseSubgr} 
	Let $G_1 \subset G$ be a dense subgroup of $G$. Then for every Tychonoff $G$-space $X$ the maximal equivariant compactifications $\beta_G X$ and $\beta_{G_1} X$ are the same. 
\end{cor}

Every compact $G$-space (with its unique uniform structure) is equiuniform.  

\begin{cor} \label{c:TotBoundqb}  
	For totally bounded uniformities we have the coincidence $\rm{EUnif}^G =  \rm{Unif}^G$. 
\end{cor}

\sk 
The following lemma with full proofs can be found only in my dissertation \cite{Me-diss85}. 
\begin{lem} \label{q-->b} \cite{Me-EqComp84, Me-diss85,Me-sing89}
	Let $X$ be a $G$-space with a topologically compatible uniformity $\U$.
	Assume that $\U$ is quasibounded. Then there exists a topologically compatible uniformity
	$\U^G \subseteq \U$ on $X$ such that $\U^G$ is bounded. Furthermore, 
	\begin{enumerate}
		\item if 
		$(X,\U) \in \rm{Unif}^G$ then $(X,\U^G) \in \rm{EUnif}^G$; 
		\item if $(X,\mu) \in \rm{EUnif}^G$ and $\mu \subset \U$ then $\mu \subset \U^G$.
		\item if $\U$ is totally bounded and $(X,\U) \in \rm{Unif}^G$, then $\U=\U^G$ (is an equiuniformity); 
		
		\item if $\U$ and $G$ are metrizable, then $\U^G$ is also metrizable;  
	\end{enumerate} 
\end{lem}
\begin{proof}
	For every $U \in N_e$ and $\eps \in \U$,  consider
	\begin{equation} \label{eq:Ueps} 
		[U,\eps] :=\{(x,y) \in X\times X: \ \ \exists u_1, u_2 \in U \ \ \ (u_1x,u_2y) \in \eps \}. 
	\end{equation}

	Then 
	$\Delta_X \subset \eps \subset [U,\eps]$ and  
	$[U_1,\eps_1] \subset [U_2,\eps_2]$ for every $U_1 \subset U_2, \eps_1 \subset \eps_2$. 
	
	It follows that  $[U_1 \cap U_2,\eps_1 \cap \eps_2] \subset [U_1,\eps_1] \cap [U_2,\eps_2]$. 
	It is also easy to see that if $U=U^{-1}$ and $\eps^{-1}=\eps$ are symmetric,  then also $[U,\eps]^{-1}=[U,\eps]$ is symmetric. 
	
	The system $\a:=\{[U,\eps]\}_{U \in N_e, \eps \in \U}$ is a filter base on the set $X \times X$.
	
	Define by $\U^G$ the corresponding filter generated by $\a$. We show that $\U^G$ is a uniformity on the set $X$. The conditions (1),(2),(3) of Lemma \ref{l:UnifBasis} are satisfied. We have to show only condition (4) for the members of the base $\a$.

	Let $[U,\eps] \in \a$. We have to show that there exists $[V,\delta]  \in \a$ such that
	$$[V,\delta] \circ [V,\delta] \subseteq [U,\eps].$$
	
	Choose $\eps_1 \in \U$ such that $\eps_1^2 \subseteq \eps$. By the quasiboundedness
	condition for $\eps_1$ there exist $\delta \in \U$ and $V \in N_e$ such that 
	\begin{equation}\label{1}
		(x,y) \in \delta, v \in V  \Rightarrow (vx,vy) \in \eps_1.
	\end{equation} 
	Without loss of generality (by properties of topological groups), we can assume in addition that
	$$V=V^{-1}, \ \ V^2 \subset U.$$
	
	We check now that $[V,\delta]^2 \subseteq [U,\eps]$.
	Let $(x,y), (y,z) \in [V,\delta]$. Then there exist $v_1,v_2,v_3,v_4 \in V$ such that $$(v_1x,v_2y), (v_3x,v_4y) \in \delta.$$
	Then by (\ref{1}) (taking into account that $V=V^{-1}$), we get
	$$
	(v_2^{-1}v_1x,y) \in \eps_1, \ \ (y,v_3^{-1}v_4z) \in \eps_1.
	$$
	
	Therefore,
	$$
	(v_2^{-1}v_1x, v_3^{-1}v_4z) \in \eps_1 \circ \eps_1 \subseteq \eps.
	$$
	
	Since $v_2^{-1}v_1$ and $v_3^{-1}v_4$ both are in $V^{-1}V \subset U$, we conclude that $(x,z) \in [U,\eps]$.

	It is easy to see the other axioms. So $\U^G$ is a uniformity.
	\sk
	
	\textbf{The uniformity $\U^G$ is topologically compatible with $X$.} That is, $top(\U^G)=top(\U)$.
	Clearly, $\eps \subseteq [U,\eps]$ for all $U \in N_e, \eps \in \U$. Hence, $\U^G \subseteq \U$.
	As to the inverse direction $\U^G \supseteq \U$, 
	one may show that for every $\eps \in \U$ and $x_0 \in X$ there exist $[V,\delta] \in \a$ such that
	$[V,\delta](x_0) \subseteq \eps(x_0)$. Indeed, using the continuity of the action, one may choose
	$\delta, \gamma \in \U$ and $V \in N_e$ such that:
	
	a) $g \ \gamma (x_0) \subset \eps(x_0) \ \ \ \forall g \in V$
	
	b) $(gx_0,x_0) \in \delta \ \ \ \forall g \in V$
	
	c) $ \delta^2 \subset \gamma, V=V^{-1}, \gamma=\gamma^{-1}$.
	
	Now, if $y \in [V,\delta](x_0)$ then $(v_1y,v_2x_0) \in \delta$ for some $v_1, v_2 \in V$. We obtain that
	$(v_1y,x_0) \in \delta \circ \delta \subset \gamma$. Hence, $v_1y \in \gamma(x_0)$. So, $y \in v_1^{-1} \gamma(x_0) \subseteq \eps(x_0)$.

	\sk
	\sk
	
	\textbf{$\U^G$ is bounded}. Indeed, $(x,ux) \in [U,\eps]$ for every $x \in X$ and every $u \in U$
	(because if we choose $u_1:=u, u_2:=e \in U$ then
	$(u_1x,u_2ux)=(ux,ux) \in \Delta_X \subseteq \eps$).

	\sk\sk
	\textbf{$\U^G $ is saturated if $\U$ is saturated.}  Let $g_0 \in G$. Then $g_0^{-1}Ug_0 \in N_e$ for every $U \in N_e$ and $g_0^{-1} \eps \in \U$ for every $\eps \in \U$ because $\U$ is saturated. Now observe that $g_0 [U,\eps]=[g_0^{-1}Ug_0,g_0^{-1}\eps]$.  
	
	\sk 
	The assertions (1), (2) and (4) easily follow now from the construction. 
	In order to check (3) consider the completion $(X,\U) \hookrightarrow (\widehat{X},\widehat{\U})$, which, in fact is a compactification because $\U$ is totally bounded. According to the completion theorem \cite{Me-ec94} we have $(\widehat{X},\widehat{\U}) \in \rm{Unif}^G$ and the action $G \times \widehat{X} \to \widehat{X}$ is continuous. Now, by Remark \ref{r:simple-qb}.1,  we obtain that $\U$ is an equiuniformity. Hence, by (2) we conclude that $\U=\U^G$.  
\end{proof}

\sk 
The equivalence of (1) and (2) in the following result is an old result which goes back at least to R. Brook \cite{Br} and J. de Vries \cite{Vr-can75}. 

\sk 
\begin{thm} \label{t:Tyc} \cite{Me-diss85} 
	Let $X$ be a Tychonoff 
	$G$-space.
	The following are equivalent: 
	\ben
	\item $X$ is $G$-Tychonoff.
	\item $(X,\U) \in \rm{Unif}^G$ for some compatible uniform structure $\U$ on $X$. 
	\item There exists a compatible uniform structure $\U$ on $X$ which is quasibounded.  
	
	\een
\end{thm}
\begin{proof} (1) $\Rightarrow$ (2) Let $X$ be $G$-Tychonoff. 
	Consider a proper $G$-compactification $\a \colon X \hookrightarrow Y$. 
	Since the natural uniformity on $Y$ is an equiuniformity, then it induces on $X$ a (precompact) equiuniformity $\U$. So, $(X,\U) \in \rm{EUnif}^S$ 
	with respect to some compatible uniformity $\U$. Now recall that $\rm{EUnif}^G \subseteq \rm{Unif}^G$.  
	
	(2) $\Rightarrow$ (3) is trivial.  
	
	(3) $\Rightarrow$ (1) 
	Let $\xi$ be a quasibounded uniformity on $X$. Then one may easily find a stronger quasibounded uniformity $\U$ which, in addition, is $G$-saturated. 
	Indeed, the system
	$\Sigma:=\{g\eps : g \in G, \eps \in \U\}$, where
	$g\eps :=\{(gx,gy) \in X \times X: \  (x,y) \in \eps\}$ 
	is a subbase of a filter of subsets in $X \times X$. Denote by $\U$ the corresponding filter generated by $\Sigma$.
	Then $\U$ is a saturated uniformity on $X$, $\xi \subseteq \U$ and $top(\xi)=top(\U)$.
	If $\Sigma$ is quasibounded or bounded, it is straightforward to show (use that the conjugations are continuous in any topological group), then $\U$ respectively is quasibounded or bounded. 
	
	Hence, we obtain $(X,\U) \in \rm{Unif}^G$. Then $(X,\xi) \in \rm{EUnif}^G$ according to Lemma \ref{q-->b}. By Proposition \ref{equiun->S-prox}, $\dl_{\U}$
	is a $G$-proximity (hence, the corresponding Smirnov's compactification is a $G$-compactification). 
\end{proof}

\begin{cor} \label{c:LudVr} \emph{(Ludescher--de Vries \cite{Lud-Vr80})}   
	Every continuous uniformly equicontinuous action of a topological group $G$ on $(X,\U)$ is $G$-Tychonoff. In particular, it is true if $X$ admits a $G$-invariant metric. 
\end{cor}


\sk 

\begin{lem} \label{l:Sigma}
	Let $\pi\colon  G \times X \to X$ be a continuous action and $\Sigma:=\{d_i\}_{i \in I}$ be a bounded system of pseudometrics on $X$ such that the induced uniform structure $\U$ on $X$ is topologically compatible. Assume that $M$ is a family of
	nonempty subsets in $G$ such that: 
	$$\forall A \in M \ \forall d_i \in \Sigma  \ \forall x,y \in X \ \ \ d_{A,i}(x,y):=sup_{g \in A} d_i(gx,gy) < \infty.$$ 
	Define by $\Sigma_M$ the system of pseudometrics $\{d_{A,i}: \ A \in M, i \in I\}$ on $X$. Let $\xi=\xi(\Sigma,M)$ be the corresponding uniform structure on $X$ generated by the system $\Sigma_M$.
	
	\ben
	\item If every $A \in M$ acts equicontinuously on $(X,\U)$, then $top(\xi)=top(\U)$.
	\item If for every $A \in M$ there exist $V \in N_e$ and $B \in M$ such that $AV \subseteq B$, then the action is $\xi$-quasibounded.
	\item If there exists $A \in M$ such that $e \in A$, then $\U \subseteq \xi$.
	\een
\end{lem}
\begin{proof}
	(1) and (3) are straightforward.
	
	(2) Observe that if $A \subset B$ then $d_{A,i}(x,y) \leq d_{B,i}(x,y)$. Therefore, if $VA \subseteq B$ then
	$$
	d_{A,i}(vx,vy) =sup_{g \in A} d(gvx,gvy) \leq  sup_{t \in AV} d(tx,ty) \leq sup_{t \in B} d(tx,ty) = d_{B,i}(x,y). 
	$$
	For every triple $A \in M, i \in I, \eps >0$ (such triples control the natural uniform subbase of $\xi$), we have 
	$$
	d_{B,i}(x,y) < \eps \Rightarrow d_{A,i}(vx,vy) < \eps \ \ \forall v \in V.
	$$ 
\end{proof}


\begin{thm} \label{l:U} \emph{(de Vries \cite{Vr-LinComp82,Vr-Revis00} and also \cite{Me-diss85})}    
	Let $G \times X \to X$ be a continuous action.
	Suppose that a neighborhood $U$ of $e$
	acts equicontinuously on $X$ with respect to some compatible uniformity $\U$. Then 
	$X$ is $G$-Tychonoff. 
\end{thm}

\begin{proof} 
	By Theorem \ref{t:Tyc},  it is enough to show that there exists a compatible finer uniformity $\xi \supseteq \U$ on the topological space $X$ which is quasibounded. 
	
	By our assumption there exists a neigborhood $U$ of $e$ in $G$ and a compatible uniformity $\xi$ on the topological space $X$
	such that $U$ acts equicontinuously on $(X,\U)$. Choose a sequence $U_n \in N_e$ such that $U_n^{-1}=U_n$, $U_{n+1}^2 \subset U_n \subset U$ for every $n \in \N$.
	Now define inductively the sequence $M:=\{V_n\}_{n \in \N}$ of subsets in $G$ where 
	$$V_n:=U_1U_2 \cdots U_n.$$
	
	Choose also a family of pseudometrics $\Sigma:=\{d_i\}_{i \in I}$ on $X$
	such that $\Sigma$ generates the uniformity $\U$. One may assume that $d_i \leq 1$ for every $i$.
	Now we define the uniformity $\xi$ as in Lemma \ref{l:Sigma} generated by the system of pseudometrics $\Sigma_M$.
\end{proof}

\sk 
For every Tychonoff space $X$ there exists the greatest compatible uniformity on $X$. We denote it 
by $\U_{max}$.

\begin{thm} \label{c:LCqB} \cite{Me-diss85}
	Let $G$ be a locally compact group. Then for every $G$-space $X$ we have $(X,\U_{max}) \in \rm{Unif}^G$.
\end{thm}
\begin{proof} According to the proof of Theorem \ref{l:U} there exists a compatible finer uniformity $\xi \supseteq \U_{max}$ on the topological space $X$ which is quasibounded. 
	Then by the maximality of $\U_{max}$ we have $\xi = \U_{max}$. Therefore $\U_{max}$ is quasibounded. In fact, again by the maximality property we obtain that $\U_{max}$ is also saturated. Hence, $(X,\U_{max}) \in \rm{Unif}^G.$  	
\end{proof}

Combining Theorems \ref{t:Tyc} and \ref{c:LCqB}, one directly gets the following well-known important result of de Vries: 

\begin{f} \label{f:VriesThm} \emph{(J. de Vries \cite{Vr-loccom78})}   
	Let $G$ be a locally compact group. Then every Tychonoff $G$-space is $G$-compactifiable. 
\end{f}



\sk 
\begin{f}  \label{f:suff} 
	Here we list several sufficient conditions of $G$-compactifiability. Some of these results were already mentioned above. 
	
	\begin{enumerate} 
	
			\item \cite{Vr-can75} Every coset $G$-space $G/H$. 
			\item \cite{Vr-can75} Every locally compact $G$-space $X$. 
			\item \cite{Vr-loccom78} Every $G$-space $X$, where $G$ is locally compact.  
		\item \cite{Me-sing89,Me-FreeTGr96}  
		Let $G$ and $X$ both be topological groups and $\a \colon G \times X \to X$ is a continuous action by group automorphisms. Then $X$ is $G$-Tychonoff (see Remark \ref{r:simple-qb}.4 and Corollary \ref{c:aut}). 
		\item 	\cite{Me-sing89}  For every metric $G$-space $(X,d)$, where $G$ is a Baire space and every $g$-translation is $d$-uniform we have $(X,\U(d)) \in \rm{Unif}^G$ (and $X$ is $G$-compactifiable). 
		\item 	\cite{Me-sing89}  If $G$ is Baire then every \textbf{metrizable} $G_{discr}$-compactification of a $G$-space $X$ is a $G$-compactification. 
		\item \cite{Usp-Dugundji} Every $G$-space $X$, where the action is algebraically transitive, $X$ is Baire and $G$ is $\aleph_0$-bounded. More generally, every d-open action. 
		\item \cite{Lud-Vr80} Every metric space $X$ with a $G$-invariant metric. More generally, every continuous uniformly equicontinuous action of a topological group $G$ on $(X,\U)$. 
		\item \label{l:U} \emph{(\cite{Vr-LinComp82,Vr-Revis00} and also \cite{Me-diss85})}    
		Let $G \times X \to X$ be a continuous action.
		Suppose that a neighborhood $U$ of $e$
		acts equicontinuously on $X$ with respect to some compatible uniformity $\U$. Then 
		$X$ is $G$-Tychonoff. 
		\item \emph{(Theorem \ref{t:OrdComp} below)} Every ordered $G_{discr}$-compactification of a $G$-space $X$ is a $G$-compactification. 
	\end{enumerate}
\end{f}
 
 \sk
 \subsection{Linearly ordered $G$-compactifications}

 In this section, 
a linearly ordered topological space (LOTS) $X$ will mean that $X$ is a topological space which topology is just the interval topology for some linear order on $X$. 
 We show that every linearly ordered $G_{discr}$-compactification of a $G$-space $X$ with the interval topology is necessarily a $G$-compactification. Recall the following result of V. Fedorchuk which gives an analog of 
 Smirnov's theorem for linearly ordered compactifications. 
 
 \begin{defin} \label{d:fed} \cite{Fed68} 
 	Let $\leq$ be a linear order on $X$.  
  	A proximity $\dl$ on $X$ is said to be an \emph{ordered proximity} 
  	(with respect to $\leq$)  if $\dl$ induces the interval topology $\tau_{\leq}$ on $X$ and the following two properties are satisfied: 
 	\begin{itemize}
 		\item [(op1)] for every $x < y$ we have $(-\infty,x] \ \overline{\dl} \  [y,+\infty)$;   		
 		\item [(op2)] for every $A \ \overline{\dl} \ B$ there exists a finite number $O_i$,  $i \in \{1,2,\cdots,n\}$ of open $\leq_X$-convex subsets \footnote{as usual, $C$ is said to be convex if $a,b \in C$ implies that the interval $(a,b)$ is a subset of $C$} 
 		such that 
 		$$
 		A \subset \cup_{i=1}^n O_i \subset X \setminus B.
 		$$ 
 	\end{itemize} 
 \end{defin}

 \begin{f} \label{f:Fed} \emph{(V. Fedorchuk \cite{Fed68})}   Let $\alpha \colon X \to Y$ be a compactification of a LOTS $X$ and 
 	 $\dl$ be the corresponding proximity on $X$. The 
 	 following conditions are equivalent:
 	 \begin{enumerate}
 	 	\item There exists a linear order $\leq_Y$ on $Y$ such that 
 	 	$Y$ is LOTS. 
 	 	\item The proximity $\dl$ of $\a$ is an ordered proximity with respect to 
 	 	the linear order $\leq_X$ on $X$ inherited from $\leq_Y$. 
 	 \end{enumerate} 
 \end{f}
 
 Note that if $\a \colon X \to Y$ is a compactification, where $(Y,\tau)$ is compact with respect to some linear order $\leq_Y$ on $Y$ (i.e., $\tau=\tau_{\leq_Y}$), then 
 the subspace topology on $X$, in general, is stronger than the interval topology of the inherited order $\leq_X$ on $X$. The coincidence $\tau_{\leq_X}=\tau|_X$ we have iff the proximity $\dl$ of $\a$ is an ordered proximity. 

 \begin{thm} \label{t:OrdComp}  
 	Let $X$ be a linearly ordered space with the interval topology of a linear order $\leq$. 
 	Let $G \times X \to X$ be a continuous action which preserves the order $\leq$. Assume that  
 	$\dl$ is an ordered proximity of a linearly ordered compactification $\a \colon (X,\leq) \to (Y,\leq_Y)$ such that $\a$ is a $G_{discr}$-compactification (i.e., $\dl$ is $G$-invariant). Then $\a$ is a $G$-compactification.  
 \end{thm}
 \begin{proof} Since $\dl$ is already $G$-invariant,  it is enough to show (by Smirnov's theorem, Fact \ref{th:Sm.thmG}) that the proximity is compatible with the action. 
 	That is, if $A \rem B$ then there exists $U \in N_e$ such that $UA \cap B =\emptyset$.  
%
%
%
 		In condition (b) of Fact \ref{f:Fed}.2, we may assume that the open convex subsets $O_i$ are disjoint. Let $A_i:=A \cap O_i$ for every $i \in \{1,2,\cdots,n\}$. Then 
 		$$A_{i_0}=A \setminus \cup_{i \neq i_0} O_i$$ 
 		is closed in $X$. 
 		Since we have finitely many $i$, it is enough to prove 
 		the following (this will cover also condition (a) in the definition of ordered proximity).  
 		
 		\sk 
 	\nt 	\textbf{Claim.} \textit{Let $A$ be a closed subset of $X$ and $O$ is an open $\leq$-convex subset of $X$ which contains $A$. Then there exists $U \in N_e$ such that $UA \subset O$. }
 		
 		\sk 
 		
 		Proof:  We may assume that $O \neq \emptyset, O \neq X$.
 		Consider other cases for convex open subsets.   
 		
 		\sk 
 		(a) $O=(-\infty,b)$. Two subcases:
 		
 	\begin{itemize}
 		\item 	[(a1)] If $[b,\infty)$ is open. 
 		By the 
 		continuity of the action, there exists $U \in N_e$ such that $Ub \subset [b,\infty)$. For every $b \leq x$ and every $u \in U$ we have $b  \leq ub  \leq  ux $ (action preserves the order). Hence, 
 		$U[b,\infty) \subset [b,\infty)$. 
 		\item [(a2)] If $[b,\infty)$ is not open. That is, $b$ is not an internal point of $[b,\infty)$. Then every open neighborhood of $b$ meets $(-\infty,b)$. Choose a 
 		neighborhood $W$ of $b$ which is disjoint with the closed set $A$. Since $X$ carries the interval topology $\tau_{\leq}$, one may assume that $W=(c,d)$. 
 		Then $a <x$ for every $a \in A$ and every $x \in (c,d)$. 
 		Choose $x_0 \in  (c,d) \cap (-\infty,b)$.  There exists  $U \in N_e$ such that $U^{-1}=U$ and  $Ux_0 \subset (c,d)$. Then $a \leq c < ux_0 <ux$ for every $x \in [b,\infty)$ and $u \in U$. So, 
 		$UA=U^{-1}A$ and $ [b,\infty)$ are disjoint.   
 	\end{itemize}

 		\sk 
 		(b) $O=(b,\infty)$. This case is completely similar to (a). 
 		
 		\sk 
 		
 		(c) $O=(-b_1,b_2)$. 
 			Combine (a), (b) taking the intersection of two neighborhoods of $e$. 
 		
 		\sk 
 		
 		(d) $O=(-\infty,b]$. 
 	 
 		\sk 
 		 Then, since $O$ is open, $b$ is an internal point of $(-\infty,b]$. 
 		There exists $U \in N_e$ such that $Ub \subset (-\infty,b]$. Then $UA \subset U(-\infty,b] \subset (-\infty,b]$. 
 		
 		\sk 
 		(e) $O=[b,\infty)$. Similar to (d).  
 		
 		\sk
 		(f) $O=[b_1,b_2]$. Combine (d) and (e).  
 \end{proof}

\sk 
\section{$G$-compactifications and proximities}

\sk 
\begin{thm} \label{t:G-proxOfQuasib} 
	Let $(X,\U) \in \rm{Unif}^G$, where $G$ is an arbitrary topological group. Then the following rule 
	\begin{equation} \label{eq:GproximityFrom} 
		A \ \nu \ B \Leftrightarrow \forall V \in N_e \ \ \ VA \ \delta_{\U} \ VB 
	\end{equation}
defines a $G$-proximity on $X$ such that $\delta_{\U^G}=\nu$ and it 
	is the greatest \textbf{$\U$-uniform} $G$-compactification of $X$ (that is, $\rho \preceq \nu$ for any $G$-proximity $\rho \preceq \dl_U$).  
\end{thm}
\begin{proof} We have to show that $\delta_{\U^G}=\nu$. That is, 
	$\nu$ coincides with the canonical proximity $\delta_{\U^G}$ of the uniformity $\U^G$ (where $\U^G$ is defined in Lemma \ref{q-->b}). 
	
	Let $A$ and $B$ be $\U^G$-far subsets in $X$. That is, $A \ \overline{\dl_{\U^G}} \ B$. Lemma \ref{q-->b} guarantees that 
	$(X,\U^G) \in \rm{EUnif}^G$.  
	Since $\U^G$ is an equiuniformity, its proximity $\dl_{\U^G}$ is a $G$-proximity by Proposition \ref{equiun->S-prox}. 
	Hence, 
	$$
	\exists \ V \in N_e \ \ \ V A \ \overline{\dl_{\U^G}} \ V B.
	$$ 
	Since $\U^G \subseteq \U$, we have $\dl_{\U^G} \preceq \dl_{\U}$. Therefore, 
	$V A \ \overline{\dl_{\U}} \  VB$. By definition of $\nu$ this means that 
	$A \ \overline{\nu} \ B$. 
	
	In the converse direction, we assume now that $A \ \overline{\nu} \ B$. That is, $V A \ \overline{\dl_{\U}} \  VB$ for some $V \in N_e$. There exists $\eps \in \U$ such that 
	$$
	\eps \cap (VA \times VB) = \emptyset.
	$$ This implies that 
	$$
	[V,\eps] \cap (A \times B) = \emptyset,
	$$
	where, as in Lemma \ref{q-->b}, $[V,\eps] =\{(x,y) \in X\times X: \  \exists v_1, v_2 \in U \ \  (v_1x,v_2y) \in \eps \}$. 
	By definition of the uniformity $\U^G$, this means that $A \ \overline{\dl_{\U^G}} \ B$. 
	  So, we can conclude that $\nu=\dl_{\U^G}$ and $\nu$ is a $G$-proximity (because, $\dl_{\U^G}$ is). 
	  
	  Since, $\U^G \subseteq \U$, we have $\nu=\dl_{\U^G} \preceq \dl_{\U}$.
	   This shows that the $G$-compactification $\nu$  of $X$ is  $\U$-uniform.  
	
	\sk 
	Finally we show the \textit{maximality property} of $\nu$. 
	Let $\rho$ be any $\U$-uniform $G$-compactification and $A \overline{\rho} B$. Then by Smirnov's theorem \ref{th:Sm.thmG} (and Remark \ref{r:semigr}), there exists $V \in N_e$ such that $VA \  \overline{\rho} \ VB$. Since $\rho$ is $\U$-uniform, we have $\rho \preceq \dl_{\U}$. Therefore,  $VA \ \overline{\dl_{\U}} \ VB$ holds. So, by Equation \ref{eq:GproximityFrom} we can conclude that 
	$VA \ \overline{\nu} \ VB$. This means that $\rho \preceq \dl_{\U^G}=\nu$. 
\end{proof}

\begin{cor} \label{c:aut}  
	Let $G$ and $X$ both be topological groups and $\a \colon G \times X \to X$ is  a continuous action by group automorphisms. Then
	the following condition 
	$$
	A \ \rem \ B \Leftrightarrow \exists V \in N_e(X) \ \exists U \in N_e(G) \ \ V \{g(A)\}_{g \in U}\ \cap  V\{g(B)\}_{g \in U} = \emptyset 
	$$
	defines a $G$-proximity which is the greatest $\mathcal{R}(X)$-uniform $G$-compactification, where $\mathcal{R}(X)$ is the right uniformity of $X$. 
\end{cor} 
\begin{proof} Observe that 
	$\lan (X,\mathcal{R}(X)),\a \ran \in \rm{Unif}^G$. 
	Now apply Theorem \ref{t:G-proxOfQuasib} to $(X,\mathcal{R}(X))$. 
\end{proof}

\sk 
\begin{defin} \label{d:domin} 
	Let $(X,d)$ be a metric space and $\pi \colon G \times X \to X$ is any action with uniform translations. We say that this action is $d$\textit{-majored} if the greatest $G$-compactification of $X$ is $d$-uniform.  
\end{defin}

\begin{prop} \label{p:metr_prox}
	Let $(X,d)$ be a metric space and  a $G$-space such that $(X,\U(d)) \in \rm{Unif}^G$ 
	(e.g., $d$ is $G$-invariant). Then the following condition 
	\begin{equation} \label{eq:maj} 
		A \ \rem \ B \Leftrightarrow \exists V \in N_e \ \ d(VA,VB) > 0
	\end{equation} 
	defines a $G$-proximity which is the greatest \textbf{$d$-uniform} $G$-compactification of $X$ (and coincides with the proximity of $\U(d)^G$). 
	
	If, in addition, the action is $d$-majored then \ref{eq:maj} describes the proximity of $\beta_G X$. 
	
	
\end{prop}
\begin{proof} 
	Let $\U(d)$ be the uniformity of the metric $d$. Its proximity $\dl_d$ 
is defined as follows:  
$$
A \dl_d B \Leftrightarrow d(A,B)=0.
$$
Now apply Theorem \ref{t:G-proxOfQuasib} to $(X,\U(d))$. 
%
	\end{proof}


\sk
\sk
\begin{remark} \label{r:Urys} 
	Let $X:=(\mathbb{U}_1,d)$ be the Urysohn sphere and $G:=\Iso (\mathbb{U}_1)$ be the Polish isometry group (pointwise topology). 
	In the joint work \cite{IbMe20} with T. Ibarlucia,  we prove that for the $G$-space $X$ the maximal $G$-compactification of $X$ is just the Gromov compactification $\gamma (\mathbb{U}_1,d)$ (in particular, $\beta_G (\mathbb{U}_1)$ is metrizable) and $\RUC(\mathbb{U}_1)$ is the unital algebra generated by the distance functions. Since $\g\colon (X,d) \to \g X$ is a $d$-uniformly continuous topological $G$-embedding, we obtain that the greatest $G$-compactification of $X$ is $d$-uniform. 
	So, the action is $d$-majored. 
%
\end{remark} 

\begin{thm} \label{t:Urys} 
	Let $X:=(\mathbb{U}_1,d)$ be the Urysohn sphere and $G:=\Iso (\mathbb{U}_1)$. Then for closed subsets $A,B$ in $\mathbb{U}_1$ we have:
	$$
	A \ \overline{\beta_G} \ B \Leftrightarrow \exists V \in N_e(G)  \ \ \  d(VA, VB) > 0.
	$$
\end{thm}
\begin{proof}
	Combine 
	Remark \ref{r:Urys} and Proposition \ref{p:metr_prox}. 
\end{proof}

\sk 
\begin{remark} \label{r:GeneralCase}  
Theorem \ref{t:Urys} remains true for a large class of all  $\aleph_0$-\textit{categorical metric $G$-structures} $M=(X,d)$ 
(including Urysohn sphere), where $G:=\Aut(M,d)$ is its automorphism group 
(for definitions, motivation and related tools, we refer to \cite{BBHU08} or, \cite{IbMe20}). In this case 
the action is $d$-majored  
as it follows from \cite[Theorem 4.4]{IbMe20}. 

Another condition which guarantees that the action is $d$-majored, 
is the \textit{uniform micro-transitivity} of the action in the sense of \cite{IbMe20}.  
\end{remark}

\sk 
%

%

\begin{thm} \label{c:MaxGproximityLocCompG} 
	\emph{(Maximal $G$-compactification for locally compact $G$ actions)} 
	  
\nt 	Let $G$ be a locally compact group. Then for every $G$-space $X$ the maximal $G$-proximity $\beta_G$ can be characterized by the maximal topological proximity $\beta$ as follows:
	$$
	A \ \beta_G \ B \Leftrightarrow \forall V \in N_e \ \ \ VA \ \beta \ VB.
	$$
	So, if $X$, as a topological space is normal, then we obtain
	$$
	A \ \beta_G \ B \Leftrightarrow \forall V \in N_e \ \ \ V cl(A) \cap V cl(B) \neq \emptyset.
	$$
\end{thm}
\begin{proof} 
	By Theorem \ref{c:LCqB} for the maximal compatible uniformity, we have 
		 $(X,\U_{max}) \in \rm{Unif}^G$. In fact, $\U_{max}$ is the greatest compatible uniformity on $X$. Therefore, 
		 its proximality defines just the usual maximal compactification $\beta X$. 
		 We use the notation $\beta$ for this proximity. Note that 
		 $VA \ \beta \ VB$ means that $VA$ and $VB$ cannot be functionally separated. 
	Now, we use Theorem \ref{t:G-proxOfQuasib} in order to complete the proof. 
	
	If $X$ is normal then the condition $VA \ \beta \ VB$ means that $cl(VA) \cap cl(VB) \neq \emptyset$ (see Example \ref{e:max0}.3). Since $G$ is locally compact, we can suppose that $V$ is compact. Hence $V M$ is closed for every closed subset $M \subset X$. In particular, we get $V cl(A)=cl(V A)$ and $V cl(B)=cl(V B)$. 
\end{proof}



\begin{remark} \label{r:ex} 
	Local compactness of $G$ is essential. Indeed, for every Polish group $G$ which is not locally compact,   there exists a second countable $G$-Tychonoff space $X$ and closed disjoint (hence, far in $\beta X$) $G$-\textbf{invariant} subsets $A,B$ such that $A \ \beta_G \ B$ (cannot be separated by $\RUC$ functions). This follows from the proof of \cite[Theorem 4.3]{MeSc98}. 
	
	Since the proof of \cite[Theorem 4.3]{MeSc98} is quite complicated, 
	we give here a concrete simpler example for the sake of completeness. 
	The idea is similar to \cite{Me-Ex88} (see also \cite{Me-b}).  
\end{remark}

\begin{example} \label{ex:ex}  
	 Let $I=[0,1]$ be the unit interval and 
	$$G_1=H_+[0,1]=\{g \in \Homeo(I): \ g(0)=0, g(1)=1\}.$$
	Denote by $\pi_1\colon G_1 \times I \to I$ the natural action of $G_1$ on $I$. 
	Then $cl(G_1 O) =[0,1]$ for every neighborhood of $0$.

	Let $\{(G_n,I_n,\pi_n): n \in \N\}$ be a countable system of ttg's, where each 
	$(G_n,I_n,\pi_n)$ is a copy of $(G_1,I,\pi_1)$. 
	Consider the special equivariant sum $(G,X,\pi)$ of the actions $\pi_n$. 
	So, $G=\prod_{k \in \N} G_k$ and $X=\bigsqcup_{k \in \N} I_k$. Clearly, $G$ is a Polish group and $X$ is a separable metrizable $G$-space.  
	Define two naturally defined subsets $A,B$ of $X$, where $A:=\{i_n(0): n \in \in \N\}$ is the set of all left end-points and $B:=\{i_n(1): n \in \N\}$ is the set of all right end-points. Then $X$, being locally compact, clearly is $G$-Tychonoff. The subsets $A,B$ are  closed disjoint and $G$-invariant subsets   in $X$. So, 
	$
	GA \cap GB = \emptyset.
	$
	 However, they cannot separated by any $\RUC$ function. Hence we have $A \ \beta_G \ B$. 
\end{example}



\sk 
\subsection*{Equivariant normality} 

Two subsets $A,B$ of a $G$-space $X$ are said to be $\pi$-\textit{disjoint} if 
$UA \cap UB = \emptyset$ for some $U \in N_e$. It is equivalent to require that 
$VA \cap B = \emptyset$ for some $V \in N_e$. 

We introduced the following definition in order to answer some questions of Yu.M. Smirnov. Among others, to have a generalized \textit{Urysohn Lemma} and a  \textit{Tietze extension theorem} for $G$-spaces. 

\begin{defin} \label{d:equinormal} \cite{Me-EqNorm83,Me-diss85} Let $G$ be a  topological group $G$. A $G$-space 
	$X$ is $G$-\textit{normal} (or, \textit{equinormal}) if for every pair of $\pi$-disjoint closed subsets $A_1,A_2$ in $X$ there exists a pair of disjoint neighborhoods $O_1,O_2$ such that $O_1$ and $O_2$ are $\pi$-disjoint. 
\end{defin} 

This concept is closely related to $G$-proximities as the following result shows. 

\begin{fact}
	The following are equivalent:
	\begin{enumerate}
		\item $X$ is $G$-normal; 
		\item every pair of $\pi$-disjoint closed subsets $A_1,A_2$ in $X$ can be separated by a function $f \in \RUC(X)$;
		\item 	the relation
		$$
		A \delta_{\pi} B  \Leftrightarrow \forall \  V \in N_e \ \ \  V cl(A) \cap V cl(B) \neq
		\emptyset
		$$
		is a proximity on $X$; 
		\item the relation from (2) is a $G$-proximity on $X$; 
		\item the relation from (2) is the maximal $G$-proximity $\beta_G$ on $X$. 
	\end{enumerate}
\end{fact}

Every $G$-normal is $G$-Tychonoff. The action of $G:=\Q$ on $X:=\R$ is not $G$-normal. By \cite{MeSc98}, $G$ is locally compact  if and only if every normal $G$-space is $G$-normal. For some additional properties of $G$-normality, we refer to \cite{NAnt17} and \cite{Me-b}.

\sk \sk 
\subsection*{Description of $\beta_G X$ by filters}

\begin{remark}
	For every proximity space $(X,\delta)$ there exists the \textit{Smirnov's compactification }
	$$s_{\delta} \colon X \to sX,$$ where $sX$ is the set of all $\delta$-\textit{ends} (\textit{maximal centered $\delta$-systems}) $\xi$.
	See \cite{Smirnov52} for details. 
	Recall that $\delta$-\textit{system} means that every member $A \in \xi$ is a $\delta$-\textit{neighborhood} of some $B \in \xi$. That is,   
	$B\Subset A$ holds 
	(meaning that 
	$B$ and $A^c:=X \setminus A$ are $\delta$-far). 
	
	Let us apply this to the case of $\beta_G$ for a $G$-space $X$. 
	
	$\bullet$ 
	If  
	$G$ is locally compact then by Corollary \ref{c:MaxGproximityLocCompG}, 
	$B \ \overline{\beta_G} \ A^c$ if and only if $VB$ and  $V A^c$ 
	are functionally separated for some $V \in N_e$. 
	
	%
	
	$\bullet$ 
	Similar results for an arbitrary topological group $G$ is not true in general. 
	However, it is true if, in addition, the action is $G$-normal in the sense of Definition \ref{d:equinormal}.
	
	$\bullet$ If $(M,d)$ is an $\aleph_0$-categorical metric structure, then by Remark \ref{r:GeneralCase} \ 
	$B \ \overline{\beta_G} \ A^c$ if and only if $d(V A^c,VB) >0$ for some $V \in N_e(\Aut(M))$. 
\end{remark}


\sk 
\section{Additional notes about maximal $G$-compactifications}


\subsection{Coset spaces and the greatest ambit}  

	For every coset $G$-space $X:=G/H$, the standard \textit{right uniformity} $\U_r$ (see, for example, \cite{RD}) is the largest possible topologically compatible equiuniformity on the $G$-space $G/H$. So, $G/H$ is a $G$-Tychonoff space (de~Vries \cite{Vr-can75}). 
Moreover, the Samuel compactification of the right uniform space $(G/H,\U_r)$ is the greatest  (proper) $G$-compactification. 
In particular, every topological group $G$ is $G$-compactifiable with respect to the standard left action (Brook \cite{Br}). 
This $G$-space $\beta_G G$ is just the \textit{greatest ambit} of $G$ which is  widely used in topological dynamics.

Let us compare this $G$-compactification with the usual topological greatest compactification $G \to \beta G$. 
The translations in this case are continuous. So, it is a $G_{discr}$-compactification.


\begin{prop} \label{ex:beta} 
Let $G$ be a metrizable topological group which is not precompact. 
Then the canonical action $G \times \beta G \to \beta G$ is continuous if and only if $G$ is discrete. 
\end{prop}
\begin{proof} 
Let $G$ is not discrete. We have to show that there exists a pair of closed disjoint subsets $A,B$ which are near \wrt right uniformity $\U_R$. 
Since $G$ is not precompact, there exists an infinite uniformly $\U_R$-discrete sequence $\{a_n\}_{n=1}^{\infty}$. This means that $$\exists U_0 \in N_e \ U_0 x_n \cap U_0x_m =\emptyset \ \ \forall m \neq n.$$ Choose a symmetric neighborhood $V \in N_e$ such that $V^2 \subset U_0$. 
Since $G$ is metrizable and not discrete, one may choose a sequence $v_n \in V$ such that $\lim v_n =e$ and all members of this sequence are distinct.  
Define 
$A:=\{a_n\}_{n=1}^{\infty}$, $B:=\{v_na_n\}_{n=1}^{\infty}.$ Then 
$UA \cap B \neq \emptyset$ for all $U \in N_e$. Therefore, $A \beta_G B$.  On the other hand, 
 $A$ and $B$ are closed disjoint subsets in the normal space $G$. Hence, $A \overline{\beta} B$. 
\end{proof}

Note that if $G$ is a pseudocompact group then $\beta G$ is a topological group 
naturally containing $G$ (see \cite{CR}). So, in this case, $G \times \beta G \to \beta G$ is continuous and $\beta_G G=\beta G$.

\sk \sk 
\subsection*{Massive actions}


\begin{defin} \label{d:massive} 
	Let $\pi\colon  G \times X \to X$ be an action of a topological group $G$ on a uniform space $(X,\U)$. 
We say that the action is $\U$-\textit{massive} if the uniform structure  $\U^G$ 
(from Lemma \ref{q-->b}) is totally bounded. 
\end{defin}

\begin{prop} \label{p:massive} 
	Let $(X,\U) \in \rm{Unif}^G$. Consider the following conditions: 
	\begin{enumerate}
		\item the greatest $\U$-uniform $G$-compactification (induced by the proximity 	$\delta_{\U^G}$) 
		of $X$ is metrizable;  
		\item  the action is $\U$-massive. 
	\end{enumerate}
Then  always (1) $\Rightarrow$ (2). If, in addition, the uniformity $\U$ 
is metrizable and $G$ is a metrizable topological group, then (2) $\Rightarrow$ (1).
\end{prop}
\begin{proof}  
	(1) $\Rightarrow$ (2): 
	By Theorem \ref{t:G-proxOfQuasib}, the greatest $\U$-uniform $G$-compactification of $X$ is the Smirnov compactification of the proximity 
	$\delta_{\U^G}$ (which is the same as the Samuel compactification of $\U^G$). 
	
	 Assume the contrary that $\U$ is not $G$-massive. Then by Definition \ref{d:massive}, $\U^G$ is not totally bounded. Equivalently, $X$ contains an infinite sequence 
	 which is $\U^G$-uniformly discrete. This implies that the corresponding Samuel compactification of $\U^G$ is not metrizable.

%
	
	(2) $\Rightarrow$ (1): Since $\U$ and $G$ are metrizable, then $\U^G$ is also metrizable. By (2), $\U^G$ is totally precompact. Then its completion is metrizable. On the other hand, this completion is the greatest $\U$-uniform $G$-compactification by Theorem \ref{t:G-proxOfQuasib}. 
\end{proof}

\sk 
Many naturally defined uniform structures are $G$-massive as it follows from the examples of Remark \ref{r:smallGcomp} making use of Proposition \ref{p:massive}.

An extreme (but useful) sufficient condition is the case of the discrete uniform space $(X,\U_{discr})$. 
%
Let us say that the action is \textit{strongly $G$-massive} if 
for every finite subset $F \subset X$ the stabilizer subgroup action $St_F \times X \to X$ has finitely many orbits.  

%
%

\begin{examples} \label{ex:massive} 
	Here we give some examples of strongly $G$-massive actions. 
	\begin{itemize}
	\item [(a)] $S_{\infty} \times \N \to \N$.  In this case $\beta_G X$ 
	(of the discrete space $X:=\N$ with the action of the Polish symmetric group $G:=S_{\infty}$) is the Alexandrov compactification $\N \cup \{\infty\}$. 
	\item [(b)] $X=(\Q,\leq)$ the rationals with the usual order but equipped with the discrete topology. Consider the automorphism group $G=\Aut(\Q,\leq)$ with the pointwise topology. In this case the action of $G$ on the discrete uniform space $(X,\U_{discr})$ is $G$-massive. Hence, $\beta_G X$ is  metrizable by Proposition \ref{p:massive}. In fact, one may show that $X \to \beta_G X$  is a proper $G$-compactification such that $\beta_G X$ is a linearly ordered $G$-space. 
	By Corollary \ref{c:DenseSubgr}, the same is true for every dense subgroup $G$ of $\Aut(\Q,\leq)$ (for instance, Thompson's group $F$). 
	\sk 
	\textbf{Sketch:} We use an idea of \cite{Me-ord75}.    
	Let $F:=\{t_1 < t_2 < \cdots < t_m\}$
	be a finite chain in $\Q$. Using the ultrahomogeneity of the action 
	$\Aut(\Q,\leq) \times (\Q,\leq) \to (\Q,\leq)$, the corresponding (finite) orbit of the stabilizer subgroup $St_F$ is 
	$$
	X_F:=\{(-\infty,t_1), t_1, (t_1,t_2), t_2, (t_2,t_3), \cdots, (t_{m-1}, t_m), , t_m, (t_m,\infty)\}. 
	$$ 
	Therefore, the present action is strongly $G$-massive and in particular 
	$\U_{discr}$-massive  ($\U_{discr}^G$ is totally bounded).  
	The proximity of the uniformity $\U_{discr}^G$ corresponds to 
	$\beta_G X$. On the other hand, the completion of $\U_{discr}^G$ can be realized as a certain inverse limit $X_{\infty}=\underleftarrow{\lim} (X_F, I)$ of finite linearly ordered sets $X_F$, where $F \in I$ and 
	the finite orbit space $X_F:=X / St_F$ carries the natural linear order.  
	

	\item [(c)]
	A similar result is valid for the circular version of (b). Namely, for the rationals on the circle with its circular order
	$X=(\Q / \Z,\circ)$, the automorphism group $G=\Aut(\Q / \Z,\circ)$ and its dense subgroups $G$ (for instance, Thompson's circular group $T$).
	In this case $\beta_G X \setminus X$ is the universal minimal $G$-space $M(G)$. 
	\end{itemize} 
\end{examples}

\sk 
\section{Some open questions}

The following questions are still open. 

\begin{question} \emph{(S. Antonyan--M. Megrelishvili)}  
	Is it true that $\dim \beta_G G=\dim G$ for every locally compact group $G$? 
	What of $G$ is a Lie group? 
\end{question}

\begin{question} \emph{(Yu.M. Smirnov (see \cite{Me-opit07} and \cite{AAS}))} 
	Let $G=\Q$ be the topological group of all rationals. Is it true that there exists a Tychonoff $G$-space which is not $G$-Tychonof ?
\end{question}

\begin{question} \emph{(H. Furstenberg and T. Scarr (see \cite{Me-opit07,Pest-Smirnov}))} 
	Let $G \times X \to X$ be a continuous action with one orbit (that is, algebraically transitive) of a (metrizable) topological group $G$ on a (metrizable) space $X$. Is it true that $X$ is  $G$-Tychonoff? 
\end{question}



\bibliographystyle{amsplain}

\end{document}